\documentclass[leqno,12pt]{amsart}
\usepackage{url}
\usepackage{latexsym}
\usepackage{amssymb,amsmath,amsfonts}
\usepackage{mathtools}
\usepackage[utf8]{inputenc}
\usepackage[usenames,dvipsnames]{color}

\newtheorem{theorem}{Theorem}[section]

\newtheorem{lemma}[theorem]{Lemma}
\newtheorem{proposition}[theorem]{Proposition}
\newtheorem{corollary}[theorem]{Corollary}
\theoremstyle{definition}
\newtheorem{definition}[theorem]{Definition}
\newtheorem{example}{Example}

\newtheorem{remark}{Remark}

\numberwithin{equation}{section}
\usepackage[left=2.5cm,right=2.5cm,top=1.6cm,bottom=1.6cm]{geometry}

\begin{document}

\title[A viscosity-Halpern hybrid scheme]{A viscosity-Halpern hybrid scheme for countable families of equilibrium and variational inequality problems}


\maketitle
\begin{center}
 Markjoe O. Uba\\
 \vspace{.1in}
Department of Mathematical Sciences,\\ 
Northern Illinois University,\\ 
DeKalb, IL 60115, USA\\
{\tt markjoeuba@gmail.com}  
\end{center}

\begin{abstract}
Let $C$ be a nonempty closed and convex subset of a uniformly smooth and uniformly convex real Banach space $E$ with dual space $E^{*}$. We introduce a viscosity-Halpern hybrid projection scheme for approximating a common element of the fixed point set of a countable family of generalized nonexpansive-type mappings, the solution sets of countably many variational inequality problems, and the solution sets of countably many equilibrium problems. The method combines a viscosity perturbation generated by a contraction, a Halpern anchor term, equilibrium and variational inequality resolvent steps, and a shrinking generalized projection step. Under monotonicity, continuity, closedness and NST-type assumptions, we prove strong convergence of the generated sequence to the generalized projection of the initial point onto the common solution set. We also give a generalized-projection variational characterization of the selected limit, residual convergence, Hilbert-space specializations, and examples showing that the full countable problem cannot, in general, be recovered from finite truncations.
\end{abstract}

\maketitle

\section{Introduction}

Let $E$ be a real Banach space with topological dual $E^{*}$ and let $C$ be a nonempty closed and convex subset of $E$. The variational inequality problem is concerned with finding a point $x^{*}\in C$ such that
\begin{equation}\label{VIintro}
\langle y-x^{*},Ax^{*}\rangle\geq 0,\qquad \forall y\in C,
\end{equation}
where $A:C\to E^{*}$ is a nonlinear mapping. We denote the set of solutions of \eqref{VIintro} by $VI(C,A)$. Variational inequality problems are important in nonlinear analysis because they include, as special cases, complementarity problems, convex minimization problems, nonlinear operator equations and several equilibrium models.

Another related problem is the equilibrium problem. Let $f:JC\times JC\to \mathbb{R}$ be a bifunction, where $J$ is the normalized duality map. The equilibrium problem considered in this paper is to find a point $x^{*}\in C$ such that
\begin{equation}\label{EPintro}
f(Jx^{*},Jy)\geq 0,\qquad \forall y\in C.
\end{equation}
The set of solutions of \eqref{EPintro} will be denoted by $EP(f)$. This problem provides a useful framework for optimization problems, saddle point problems, variational inequality problems and fixed point problems.

Fixed point theory is another central topic in nonlinear functional analysis. For a self map $T:C\to C$, the fixed point problem is to find $x^{*}\in C$ such that $Tx^{*}=x^{*}$. In Banach spaces, however, one also encounters mappings from a space into its dual. For this reason, the notion of $J$-fixed points was introduced. If $T:C\to E^{*}$, a point $p\in C$ is called a $J$-fixed point of $T$ if $Tp=Jp$. This notion makes it possible to study fixed point-type problems for non-self mappings $T:C\to E^{*}$.

In recent years, hybrid projection methods have been widely used to obtain strong convergence theorems for common solution problems. These methods are especially useful because many nonexpansive-type algorithms in infinite-dimensional spaces yield weak convergence only. In the Banach space setting, the Lyapunov functional
\begin{equation}\label{phi}
\phi(x,y)=\|x\|^{2}-2\langle x,Jy\rangle+\|y\|^{2},\qquad x,y\in E.
\end{equation}
plays the role of the squared norm in Hilbert spaces. In particular, if $E=H$ is a real Hilbert space, then $J=I$ and $\phi(x,y)=\|x-y\|^{2}$.

Hybrid and monotone projection methods for fixed point problems, convex feasibility problems, equilibrium problems, variational inequalities and monotone-type inclusions have been studied extensively; see, for example, \cite{Alber1996,AlberRyazantseva2006,BlumOettli1994,ChidumeIdu2016,ChidumeOtuboEzeaUba2017,Cioranescu1990,CombettesHirstoaga2005,IbarakiTakahashi2007,KamimuraTakahashi2002,KlineamSuantaiTakahashi2012,KohsakaTakahashi2007,NakajoTakahashi2003,QinSu2008,Takahashi2000,TakahashiZembayashi2008,UbaOtuboOnyido2021,ZegeyeShahzad2011,ZegeyeShahzad2014}. In \cite{UbaCarpathian2023}, the authors considered a hybrid scheme for approximating a common element of the set of $J$-fixed points of a countable family of generalized $J_{*}$-nonexpansive maps together with finite families of variational inequality and equilibrium problems. The present work continues this line of research in two directions: we add a viscosity-Halpern regularization term, and we replace the finite equilibrium and variational inequality families by countable families.

The Halpern iteration, introduced by Halpern \cite{Halpern1967}, is a powerful device for forcing strong convergence by combining the current iterate with a fixed anchor. Viscosity methods, introduced and developed by Moudafi \cite{Moudafi2000} and further studied by Xu \cite{Xu2004}, use a contraction to regularize the limiting behavior of a nonexpansive-type process. These two ideas have become standard tools in the approximation of fixed points and solutions of variational inequalities.

The countable setting introduces an additional indexing issue: every component problem must be visited infinitely often while the shrinking projection sets remain nested and convex. More precisely, before computing the equilibrium and variational inequality resolvent steps, we form the viscosity-Halpern regularized point
\[
s_n=J^{-1}\big(\lambda_nJh(x_n)+\sigma_nJ\bar u+(1-\lambda_n-\sigma_n)Jx_n\big),
\]
where $h:C\to C$ is a contraction and $\bar u\in C$ is a fixed anchor. The equilibrium and variational inequality resolvents are then evaluated at $s_n$. The new method therefore combines viscosity perturbation, Halpern anchoring, countable-family resolvent selection and a shrinking generalized projection step.

The main contribution of this paper is a strong convergence theorem for a single iterative scheme which simultaneously treats three countable structures: a countable family of generalized $J_{*}$-nonexpansive mappings, countably many equilibrium problems and countably many variational inequality problems. The countable setting is not merely a change of notation from the finite case. It introduces an indexing challenge: every component problem must be visited infinitely often while the shrinking projection sets remain nonempty, closed, convex and nested. The index-map framework used here resolves this difficulty and permits the equilibrium and variational inequality resolvents to be selected one at a time without losing convergence to the full countable intersection. The paper also studies residual convergence, a generalized-projection variational characterization of the selected limit, and examples showing that finite truncations do not generally recover the countable solution set. When the viscosity and Halpern parameters vanish identically and the equilibrium and variational inequality families are finite, the method reduces to earlier hybrid schemes of the same type.

For clarity, Table~\ref{comparison-table} indicates the position of the present result relative to some related hybrid projection frameworks.
\begin{table}[ht]
\centering
\footnotesize
\caption{Comparison with related hybrid projection frameworks}
\label{comparison-table}
\renewcommand{\arraystretch}{1.18}
\begin{tabular}{lcccc}
\hline
Work & FP family & EP family & VI family & Setting/regularization \\
\hline
Zegeye--Shahzad \cite{ZegeyeShahzad2011,ZegeyeShahzad2014} & finite & finite & finite & hybrid framework \\
Uba et al. \cite{UbaCarpathian2023} & countable & finite & finite & Banach hybrid scheme \\
Uba--Otubo--Onyido \cite{UbaOtuboOnyido2021} & countable & single & absent & Banach hybrid method \\
Present paper & countable & countable & countable & viscosity-Halpern Banach scheme \\
\hline
\end{tabular}
\end{table}

\section{Preliminaries}

Throughout this paper, $E$ will denote a real Banach space with dual space $E^{*}$. The normalized duality map $J:E\to 2^{E^{*}}$ is defined by
\[
Jx:=\{x^{*}\in E^{*}:\langle x,x^{*}\rangle=\|x\|\|x^{*}\|,\ \|x\|=\|x^{*}\|\}.
\]
If $E$ is smooth, strictly convex and reflexive, then $J$ is single-valued, one-to-one and onto. Moreover, if $E$ is uniformly smooth, then $J$ is uniformly continuous on bounded subsets of $E$. In this setting the inverse duality mapping from $E^{*}$ into $E$ is denoted by $J_{*}=J^{-1}$.

Let $E$ be a smooth real Banach space. The Lyapunov functional $\phi:E\times E\to \mathbb{R}$ is defined by \eqref{phi}, and has the following property
\begin{equation}\label{phibounds}
(\|x\|-\|y\|)^2\leq \phi(x,y)\leq (\|x\|+\|y\|)^2,
\end{equation}
for all $x,y\in E$.

\begin{definition}\label{gprojdef}
Let $C$ be a nonempty closed and convex subset of a smooth Banach space $E$. The generalized projection of $x\in E$ onto $C$ is the unique point $\Pi_Cx\in C$ satisfying
\[
\phi(\Pi_Cx,x)=\inf_{y\in C}\phi(y,x),
\]
whenever such a point exists. In a uniformly smooth and uniformly convex Banach space, $\Pi_Cx$ exists and is unique for every nonempty closed and convex subset $C$ and every $x\in E$.
\end{definition}

\begin{definition}\label{Jfixeddef}
Let $T:C\to E^{*}$ be a map. A point $p\in C$ is called a $J$-fixed point of $T$ if
\[
Tp=Jp.
\]
The set of $J$-fixed points of $T$ will be denoted by $F_J(T)$. If $\Gamma$ is a family of mappings from $C$ into $E^{*}$, we write
\[
F_J(\Gamma):=\bigcap_{T\in\Gamma}F_J(T).
\]
\end{definition}

\begin{definition}\label{GenJstar}
A map $T:C\to E^{*}$ is called generalized $J_{*}$-nonexpansive if $F_J(T)\neq\emptyset$ and
\[
\phi(p,(J_{*}\circ T)x)\leq \phi(p,x),\qquad \forall x\in C,\ p\in F_J(T).
\]
\end{definition}

\begin{definition}\label{Jcloseddef}
A map $T:C\to E^{*}$ is called $J_{*}$-closed if $(J_{*}\circ T):C\to E$ is closed; that is, whenever $x_n\to x$ and $(J_{*}\circ T)x_n\to y$, then $(J_{*}\circ T)x=y$.
\end{definition}

\noindent {\bf NST-condition.} Let $\{T_n\}$ and $\Gamma$ be two families of generalized $J_{*}$-nonexpansive maps from $C$ into $E^{*}$ such that
\[
\bigcap_{n=1}^{\infty}F_J(T_n)=F_J(\Gamma)\neq\emptyset.
\]
The sequence $\{T_n\}$ is said to satisfy the NST-condition with $\Gamma$ if for each bounded sequence $\{x_n\}\subset C$,
\[
\|Jx_n-T_nx_n\|\to 0\quad \Longrightarrow\quad \|Jx_n-Tx_n\|\to 0,\qquad \forall T\in\Gamma.
\]

For the equilibrium problem, we assume that a bifunction $f:JC\times JC\to \mathbb{R}$ satisfies:
\begin{itemize}
\item[(A1)] $f(x^{*},x^{*})=0$ for all $x^{*}\in JC$;
\item[(A2)] $f$ is monotone, that is,
\[
f(x^{*},y^{*})+f(y^{*},x^{*})\leq 0,
\]
for all $x^{*},y^{*}\in JC$;
\item[(A3)] for all $x^{*},y^{*},z^{*}\in JC$,
\[
\limsup_{t\downarrow 0}f(tz^{*}+(1-t)x^{*},y^{*})\leq f(x^{*},y^{*});
\]
\item[(A4)] for each $x^{*}\in JC$, $f(x^{*},\cdot)$ is convex and lower semicontinuous.
\end{itemize}

We recall the following lemmas which will be used in the sequel.

\begin{lemma}\label{glemma}
Let $E$ be a uniformly convex Banach space, $r>0$ and $B_r(0)$ be the closed ball of $E$. For any given points $x_1,x_2,\ldots,x_N\in B_r(0)$ and any positive numbers $\lambda_1,\lambda_2,\ldots,\lambda_N$ with $\sum_{i=1}^{N}\lambda_i=1$, there exists a continuous, strictly increasing and convex function $g:[0,2r)\to [0,\infty)$ with $g(0)=0$ such that, for any $i,j\in\{1,2,\ldots,N\}$ with $i<j$,
\[
\left\|\sum_{n=1}^{N}\lambda_nx_n\right\|^2\leq \sum_{n=1}^{N}\lambda_n\|x_n\|^2-\lambda_i\lambda_jg(\|x_i-x_j\|).
\]
\end{lemma}

\begin{lemma}\label{phitozero}
Let $E$ be a real smooth and uniformly convex Banach space, and let $\{x_n\}$ and $\{y_n\}$ be two sequences in $E$. If either $\{x_n\}$ or $\{y_n\}$ is bounded and $\phi(x_n,y_n)\to 0$ as $n\to\infty$, then
\[
\|x_n-y_n\|\to 0.
\]
\end{lemma}

\begin{lemma}[Generalized projection inequality]\label{projineq}
Let $E$ be a smooth, strictly convex and reflexive Banach space, and let $C$ be a nonempty closed and convex subset of $E$. Then $z=\Pi_Cx$ if and only if
\[
\langle y-z,Jx-Jz\rangle\leq 0,
\qquad \forall y\in C.
\]
Moreover,
\[
\phi(y,z)+\phi(z,x)\leq \phi(y,x),
\qquad \forall y\in C.
\]
\end{lemma}

\begin{lemma}\label{EPrezolvent}
Let $C$ be a nonempty closed subset of a smooth, strictly convex and reflexive Banach space $E$ such that $JC$ is closed and convex. Let $f$ be a bifunction from $JC\times JC$ to $\mathbb{R}$ satisfying $(A1)-(A4)$. For $r>0$ and $x\in E$, define $T_r^f:E\to C$ by
\[
T_r^f x=\left\{z\in C:f(Jz,Jy)+\frac{1}{r}\langle y-z,Jz-Jx\rangle\geq 0,\ \forall y\in C\right\}.
\]
Then $T_r^f$ is single valued, $F(T_r^f)=EP(f)$, $JEP(f)$ is closed and convex, and
\[
\phi(p,T_r^f x)+\phi(T_r^f x,x)\leq \phi(p,x),\qquad \forall p\in EP(f).
\]
\end{lemma}

\begin{lemma}\label{VIrezolvent}
Let $C$ be a nonempty closed subset of a smooth, strictly convex and reflexive Banach space $E$. Let $A:C\to E^{*}$ be a continuous monotone mapping. For $r>0$ and $x\in E$, define $F_r^A:E\to C$ by
\[
F_r^A x=\left\{z\in C:\langle y-z,Az\rangle+\frac{1}{r}\langle y-z,Jz-Jx\rangle\geq 0,\ \forall y\in C\right\}.
\]
Then $F_r^A$ is single valued, $F(F_r^A)=VI(C,A)$, $JVI(C,A)$ is closed and convex, and
\[
\phi(p,F_r^A x)+\phi(F_r^A x,x)\leq \phi(p,x),\qquad \forall p\in VI(C,A).
\]
\end{lemma}

\section{Main results}

Let $E$ be a uniformly smooth and uniformly convex real Banach space with dual space $E^{*}$, and let $C$ be a nonempty closed and convex subset of $E$ such that $JC$ is closed and convex. Let $\{f_i\}_{i=1}^{\infty}$ be a countable family of bifunctions from $JC\times JC$ to $\mathbb{R}$ satisfying $(A1)-(A4)$, and let $\{A_j\}_{j=1}^{\infty}$ be a countable family of continuous monotone mappings from $C$ into $E^{*}$. Let $\{T_n\}_{n=1}^{\infty}$ be a countable family of generalized $J_{*}$-nonexpansive mappings from $C$ into $E^{*}$.

Let $\mu,\nu:\mathbb{N}\to\mathbb{N}$ be index maps such that
\[
\{n\in\mathbb{N}:\mu(n)=i\}\quad\text{and}\quad
\{n\in\mathbb{N}:\nu(n)=j\}
\]
are infinite for every $i,j\in\mathbb{N}$. For convenience, at the $n$th step we write
\[
f_n:=f_{\mu(n)},\qquad A_n:=A_{\nu(n)}.
\]
Thus each equilibrium problem and each variational inequality problem is selected infinitely many times.

Let $h:C\to C$ be a contraction with coefficient $\rho\in(0,1)$, and let $\bar u\in C$ be fixed. Let $\{\lambda_n\}$ and $\{\sigma_n\}$ be sequences in $[0,1)$ such that
\begin{equation}\label{lambdacond}
\lambda_n+\sigma_n<1,
\qquad \lambda_n\to 0,
\qquad \sigma_n\to 0.
\end{equation}
Let $\alpha_1,\alpha_2,\alpha_3\in(0,1)$ satisfy
\begin{equation}\label{alphacond}
\alpha_1+\alpha_2+\alpha_3=1.
\end{equation}
Let $\{r_n\}\subset [a,\infty)$ for some $a>0$. Starting from $x_1\in C$ and $C_1=C$, define
\begin{equation}\label{algorithm}
\left\{
\begin{array}{ll}
s_n=J^{-1}\big(\lambda_nJh(x_n)+\sigma_nJ\bar u+(1-\lambda_n-\sigma_n)Jx_n\big),\\[2mm]
z_n=\left\{z\in C:f_n(Jz,Jy)+\frac{1}{r_n}\langle y-z,Jz-Js_n\rangle\geq 0,\ \forall y\in C\right\},\\[2mm]
u_n=\left\{z\in C:\langle y-z,A_nz\rangle+\frac{1}{r_n}\langle y-z,Jz-Js_n\rangle\geq 0,\ \forall y\in C\right\},\\[2mm]
y_n=J^{-1}\big(\alpha_1Js_n+\alpha_2Jz_n+\alpha_3T_nu_n\big),\\[2mm]
C_{n+1}=\{v\in C_n:\phi(v,y_n)\leq \phi(v,s_n)\},\\[2mm]
x_{n+1}=\Pi_{C_{n+1}}x_1,
\end{array}
\right.
\end{equation}
for all $n\geq 1$.

\begin{remark}
The first line of \eqref{algorithm} is the viscosity-Halpern regularization step. The contraction $h$ gives the viscosity perturbation, while the fixed point $\bar u$ gives the Halpern anchor. If $\lambda_n=\sigma_n=0$ for all $n\geq 1$, then $s_n=x_n$ and the method reduces to a non-regularized countable-family hybrid projection scheme.
\end{remark}

\begin{remark}
A simple admissible index rule is the triangular ordering
\[
1,1,2,1,2,3,1,2,3,4,\ldots.
\]
Using this ordering for both $\mu$ and $\nu$ ensures that every positive integer is selected infinitely many times.
\end{remark}

\begin{theorem}\label{mainthm}
Let $E$ be a uniformly smooth and uniformly convex real Banach space with dual space $E^{*}$, and let $C$ be a nonempty closed and convex subset of $E$ such that $JC$ is closed and convex. Let $\{f_i\}_{i=1}^{\infty}$ be a countable family of bifunctions from $JC\times JC$ to $\mathbb{R}$ satisfying $(A1)-(A4)$, let $\{A_j\}_{j=1}^{\infty}$ be a countable family of continuous monotone mappings from $C$ into $E^{*}$, and let $\{T_n\}_{n=1}^{\infty}$ be a countable family of generalized $J_{*}$-nonexpansive maps. Let $\Gamma$ be a family of $J_{*}$-closed and generalized $J_{*}$-nonexpansive maps from $C$ into $E^{*}$ such that
\[
\bigcap_{n=1}^{\infty}F_J(T_n)=F_J(\Gamma)\neq\emptyset.
\]
Assume that $\{T_n\}$ satisfies the NST-condition with $\Gamma$, and suppose that the common solution set
\[
B:=F_J(\Gamma)\cap\left[\bigcap_{i=1}^{\infty}EP(f_i)\right]\cap\left[\bigcap_{j=1}^{\infty}VI(C,A_j)\right]
\]
is nonempty, closed and convex. Then the sequence $\{x_n\}$ generated by \eqref{algorithm} is well defined and converges strongly to $\Pi_Bx_1$, the generalized projection of $x_1$ onto $B$.
\end{theorem}

\begin{proof}
We divide the proof into seven steps.

\noindent {\bf Step 1.} We show that the construction is well defined and that $B\subset C_n$ for all $n\geq 1$.

Clearly, $B\subset C_1=C$. Suppose that $B\subset C_n$ for some $n\geq 1$. Let $p\in B$. Then $p\in EP(f_n)$ and $p\in VI(C,A_n)$. By Lemmas~\ref{EPrezolvent} and \ref{VIrezolvent}, and by the definitions of $z_n$ and $u_n$, we have
\begin{equation}\label{resolventineq}
\phi(p,z_n)\leq \phi(p,s_n),\qquad \phi(p,u_n)\leq \phi(p,s_n).
\end{equation}
Since $p\in F_J(T_n)$ and $T_n$ is generalized $J_{*}$-nonexpansive,
\begin{equation}\label{Tineq}
\phi(p,(J_{*}\circ T_n)u_n)\leq \phi(p,u_n)\leq \phi(p,s_n).
\end{equation}
Using the definition of $y_n$ and the convexity inequality for the square of the norm in $E^{*}$, we obtain
\begin{align*}
\phi(p,y_n)
&=\phi\big(p,J^{-1}(\alpha_1Js_n+\alpha_2Jz_n+\alpha_3T_nu_n)\big)\\
&\leq \alpha_1\phi(p,s_n)+\alpha_2\phi(p,z_n)+\alpha_3\phi(p,(J_{*}\circ T_n)u_n)\\
&\leq \phi(p,s_n).
\end{align*}
Thus $p\in C_{n+1}$, and hence $B\subset C_{n+1}$.

Moreover, $C_{n+1}$ is closed and convex because the inequality
\[
\phi(v,y_n)\leq \phi(v,s_n)
\]
is equivalent to
\[
2\langle v,Js_n-Jy_n\rangle\leq \|s_n\|^2-\|y_n\|^2,
\]
which defines a closed half-space intersected with $C_n$. Since $B\subset C_{n+1}$, the set $C_{n+1}$ is nonempty. By induction, each $C_n$ is nonempty, closed and convex. Therefore $\Pi_{C_n}x_1$ exists for each $n$, and the algorithm is well defined.

\noindent {\bf Step 2.} We show that $\{x_n\}$ converges strongly to some point $x^{*}\in C$.

Since $x_n=\Pi_{C_n}x_1$ and $B\subset C_n$, we have
\[
\phi(x_n,x_1)\leq \phi(p,x_1),\qquad \forall p\in B.
\]
Thus $\{\phi(x_n,x_1)\}$ is bounded, and consequently $\{x_n\}$ is bounded. Also, since $x_{n+1}\in C_{n+1}\subset C_n$ and $x_n=\Pi_{C_n}x_1$, we have
\[
\phi(x_n,x_1)\leq \phi(x_{n+1},x_1),\qquad n\geq 1.
\]
Therefore $\lim_{n\to\infty}\phi(x_n,x_1)$ exists.

Let $m>n$. Since $x_m\in C_m\subset C_n$, Lemma~\ref{projineq} gives
\begin{align}\label{Cauchyineq}
\phi(x_m,x_n)&\leq \phi(x_m,x_1)-\phi(x_n,x_1)\to 0
\end{align}
as $m,n\to\infty$. By Lemma~\ref{phitozero}, $\|x_m-x_n\|\to 0$ as $m,n\to\infty$. Hence $\{x_n\}$ is Cauchy. Since $C$ is closed, there exists $x^{*}\in C$ such that
\begin{equation}\label{xnconv}
x_n\to x^{*}.
\end{equation}

\noindent {\bf Step 3.} We show that $s_n\to x^{*}$ and $y_n\to x^{*}$.

Since $\{x_n\}$ is bounded and $h$ is a contraction, the sequence $\{h(x_n)\}$ is bounded. Also $\bar u$ is fixed. By \eqref{lambdacond},
\[
\lambda_nJh(x_n)+\sigma_nJ\bar u+(1-\lambda_n-\sigma_n)Jx_n-Jx_n\to 0.
\]
Since $J^{-1}=J_{*}$ is uniformly continuous on bounded subsets of $E^{*}$, we obtain
\begin{equation}\label{snconv}
\|s_n-x_n\|\to 0.
\end{equation}
Combining \eqref{xnconv} and \eqref{snconv}, we get
\begin{equation}\label{snxstar}
s_n\to x^{*}.
\end{equation}

Since $x_{n+1}\in C_{n+1}$, we have
\[
\phi(x_{n+1},y_n)\leq \phi(x_{n+1},s_n).
\]
Using \eqref{xnconv} and \eqref{snxstar}, we obtain $\phi(x_{n+1},s_n)\to 0$. Hence $\phi(x_{n+1},y_n)\to 0$. By Lemma~\ref{phitozero},
\[
\|x_{n+1}-y_n\|\to 0.
\]
Since $x_{n+1}\to x^{*}$, it follows that
\begin{equation}\label{ynconv}
y_n\to x^{*}.
\end{equation}

\noindent {\bf Step 4.} We show that $u_n\to x^{*}$, $z_n\to x^{*}$ and $\|Ju_n-T_nu_n\|\to 0$.

Let $p\in B$. The sequences $\{Js_n\}$, $\{Jz_n\}$ and $\{T_nu_n\}$ are bounded in $E^{*}$. Applying Lemma~\ref{glemma} in $E^{*}$ to the bounded ball containing these points, we obtain a continuous, strictly increasing and convex function $g$ with $g(0)=0$ such that
\begin{align}\label{gineq}
\phi(p,y_n)
&\leq \alpha_1\phi(p,s_n)+\alpha_2\phi(p,z_n)+\alpha_3\phi(p,(J_{*}\circ T_n)u_n)\nonumber\\
&\quad -\alpha_1\alpha_3g(\|Js_n-T_nu_n\|)\nonumber\\
&\leq \phi(p,s_n)-\alpha_1\alpha_3g(\|Js_n-T_nu_n\|).
\end{align}
Since $s_n\to x^{*}$ and $y_n\to x^{*}$, we have
\[
\phi(p,s_n)-\phi(p,y_n)\to 0.
\]
It follows from \eqref{gineq} that
\[
g(\|Js_n-T_nu_n\|)\to 0.
\]
Since $g$ is strictly increasing and $g(0)=0$, we get
\begin{equation}\label{residual1}
\|Js_n-T_nu_n\|\to 0.
\end{equation}

Furthermore, from \eqref{gineq} and the fact that each term in the convex combination is not larger than $\phi(p,s_n)$, we obtain
\begin{equation}\label{conv_phi_terms}
\phi(p,z_n)\to \phi(p,x^{*}),\qquad \phi(p,u_n)\to \phi(p,x^{*}).
\end{equation}
Using Lemmas~\ref{EPrezolvent} and \ref{VIrezolvent}, we have
\[
\phi(z_n,s_n)\leq \phi(p,s_n)-\phi(p,z_n),
\]
and
\[
\phi(u_n,s_n)\leq \phi(p,s_n)-\phi(p,u_n).
\]
By \eqref{snxstar} and \eqref{conv_phi_terms}, the right hand sides tend to zero. Thus
\[
\phi(z_n,s_n)\to 0,
\qquad
\phi(u_n,s_n)\to 0.
\]
By Lemma~\ref{phitozero},
\begin{equation}\label{unznsn}
\|z_n-s_n\|\to 0,
\qquad
\|u_n-s_n\|\to 0.
\end{equation}
Combining \eqref{snxstar} and \eqref{unznsn}, we obtain
\begin{equation}\label{unzconv}
z_n\to x^{*},
\qquad
u_n\to x^{*}.
\end{equation}
Since $J$ is uniformly continuous on bounded subsets of $E$, \eqref{unznsn} implies
\[
\|Ju_n-Js_n\|\to 0.
\]
Together with \eqref{residual1}, we get
\begin{equation}\label{residual2}
\|Ju_n-T_nu_n\|\leq \|Ju_n-Js_n\|+\|Js_n-T_nu_n\|\to 0.
\end{equation}

\noindent {\bf Step 5.} We prove that $x^{*}\in F_J(\Gamma)$.

From \eqref{unzconv}, $u_n\to x^{*}$. From \eqref{residual2},
\[
\|Ju_n-T_nu_n\|\to 0.
\]
Since $\{T_n\}$ satisfies the NST-condition with $\Gamma$, we obtain
\begin{equation}\label{NSTuse}
\|Ju_n-Tu_n\|\to 0,
\qquad \forall T\in \Gamma.
\end{equation}
Let $T\in\Gamma$. Since $u_n\to x^{*}$ and $T$ is $J_{*}$-closed, \eqref{NSTuse} implies that $Tx^{*}=Jx^{*}$. Hence $x^{*}\in F_J(T)$ for every $T\in\Gamma$. Therefore
\begin{equation}\label{fixedmembership}
x^{*}\in F_J(\Gamma).
\end{equation}

\noindent {\bf Step 6.} We prove that
\[
x^{*}\in \left[\bigcap_{i=1}^{\infty}EP(f_i)\right]\cap\left[\bigcap_{j=1}^{\infty}VI(C,A_j)\right].
\]

We first show that $x^{*}\in \bigcap_{j=1}^{\infty}VI(C,A_j)$. From the definition of $u_n$,
\begin{equation}\label{VIbasic}
\langle y-u_n,A_nu_n\rangle+\frac{1}{r_n}\langle y-u_n,Ju_n-Js_n\rangle\geq 0,
\qquad \forall y\in C.
\end{equation}
Since $r_n\geq a>0$ and $\|Ju_n-Js_n\|\to 0$, we have
\begin{equation}\label{VIfraction}
\frac{\|Ju_n-Js_n\|}{r_n}\to 0.
\end{equation}
Fix $j\in\mathbb{N}$. Since $\nu^{-1}(j)$ is infinite, choose a subsequence $\{n_m\}$ such that $A_{n_m}=A_j$ for all $m\geq 1$. From \eqref{VIbasic},
\[
\langle y-u_{n_m},A_ju_{n_m}\rangle+\frac{1}{r_{n_m}}\langle y-u_{n_m},Ju_{n_m}-Js_{n_m}\rangle\geq 0,
\qquad \forall y\in C.
\]
For $t\in(0,1]$ and $y\in C$, set $v_t=ty+(1-t)x^{*}$. Since $C$ is convex, $v_t\in C$. Using monotonicity of $A_j$, we obtain
\[
\langle v_t-u_{n_m},A_jv_t\rangle\geq -\frac{1}{r_{n_m}}\langle v_t-u_{n_m},Ju_{n_m}-Js_{n_m}\rangle.
\]
Letting $m\to\infty$ and using \eqref{unzconv} and \eqref{VIfraction}, we get
\[
\langle v_t-x^{*},A_jv_t\rangle\geq 0.
\]
Since $v_t-x^{*}=t(y-x^{*})$, it follows that
\[
\langle y-x^{*},A_jv_t\rangle\geq 0.
\]
Letting $t\downarrow 0$ and using the continuity of $A_j$, we obtain
\[
\langle y-x^{*},A_jx^{*}\rangle\geq 0,
\qquad \forall y\in C.
\]
Thus $x^{*}\in VI(C,A_j)$. Since $j\in\mathbb{N}$ was arbitrary,
\begin{equation}\label{VIall}
x^{*}\in \bigcap_{j=1}^{\infty}VI(C,A_j).
\end{equation}

We now show that $x^{*}\in \bigcap_{i=1}^{\infty}EP(f_i)$. From the definition of $z_n$,
\begin{equation}\label{EPbasic}
f_n(Jz_n,Jy)+\frac{1}{r_n}\langle y-z_n,Jz_n-Js_n\rangle\geq 0,
\qquad \forall y\in C.
\end{equation}
Since $r_n\geq a>0$ and $\|Jz_n-Js_n\|\to 0$, we have
\begin{equation}\label{EPfraction}
\frac{\|Jz_n-Js_n\|}{r_n}\to 0.
\end{equation}
Fix $i\in\mathbb{N}$. Since $\mu^{-1}(i)$ is infinite, choose a subsequence $\{m_q\}$ such that $f_{m_q}=f_i$ for all $q\geq 1$. From \eqref{EPbasic} and the monotonicity condition $(A2)$,
\[
\frac{1}{r_{m_q}}\langle y-z_{m_q},Jz_{m_q}-Js_{m_q}\rangle\geq -f_i(Jz_{m_q},Jy)\geq f_i(Jy,Jz_{m_q}).
\]
Letting $q\to\infty$ and using \eqref{unzconv}, \eqref{EPfraction} and the lower semicontinuity of $f_i(Jy,\cdot)$, we obtain
\[
f_i(Jy,Jx^{*})\leq 0,
\qquad \forall y\in C.
\]
For $t\in(0,1]$ and $y\in C$, put
\[
y_t^{*}=tJy+(1-t)Jx^{*}.
\]
Since $JC$ is convex, $y_t^{*}\in JC$. Hence $f_i(y_t^{*},Jx^{*})\leq 0$. By $(A1)$ and the convexity of the second argument,
\[
0=f_i(y_t^{*},y_t^{*})\leq t f_i(y_t^{*},Jy)+(1-t)f_i(y_t^{*},Jx^{*})\leq t f_i(y_t^{*},Jy).
\]
Thus $f_i(y_t^{*},Jy)\geq 0$. Letting $t\downarrow 0$ and using $(A3)$, we obtain
\[
f_i(Jx^{*},Jy)\geq 0,
\qquad \forall y\in C.
\]
Therefore $x^{*}\in EP(f_i)$. Since $i\in\mathbb{N}$ was arbitrary,
\begin{equation}\label{EPall}
x^{*}\in \bigcap_{i=1}^{\infty}EP(f_i).
\end{equation}
Combining \eqref{fixedmembership}, \eqref{VIall} and \eqref{EPall}, we have
\begin{equation}\label{xinstarB}
x^{*}\in B.
\end{equation}

\noindent {\bf Step 7.} We show that $x^{*}=\Pi_Bx_1$.

Since $B$ is nonempty, closed and convex, $\Pi_Bx_1$ exists. By the defining property of the generalized projection,
\begin{equation}\label{PBineq1}
\phi(\Pi_Bx_1,x_1)\leq \phi(x^{*},x_1).
\end{equation}
Since $B\subset C_n$ for every $n\geq 1$ and $x_n=\Pi_{C_n}x_1$, we also have
\[
\phi(x_n,x_1)\leq \phi(\Pi_Bx_1,x_1),
\qquad n\geq 1.
\]
Letting $n\to\infty$, we obtain
\begin{equation}\label{PBineq2}
\phi(x^{*},x_1)\leq \phi(\Pi_Bx_1,x_1).
\end{equation}
From \eqref{PBineq1} and \eqref{PBineq2},
\[
\phi(x^{*},x_1)=\phi(\Pi_Bx_1,x_1).
\]
Since $x^{*}\in B$ and the generalized projection is unique, we conclude that
\[
x^{*}=\Pi_Bx_1.
\]
This completes the proof.
\end{proof}

\section{Applications and consequences}

\begin{proposition}[Residual convergence]\label{residualprop}
Under the assumptions of Theorem~\ref{mainthm}, the auxiliary sequences generated by \eqref{algorithm} satisfy
\[
\|s_n-x_n\|\to 0,
\qquad
\|y_n-x_{n+1}\|\to 0,
\qquad
\|z_n-s_n\|\to 0,
\qquad
\|u_n-s_n\|\to 0,
\]
and
\[
\|Ju_n-T_nu_n\|\to 0.
\]
Consequently,
\[
s_n\to \Pi_Bx_1,
\qquad y_n\to \Pi_Bx_1,
\qquad z_n\to \Pi_Bx_1,
\qquad u_n\to \Pi_Bx_1.
\]
\end{proposition}

\begin{proof}
These conclusions are precisely the residual estimates established in Steps 3 and 4 of the proof of Theorem~\ref{mainthm}.
\end{proof}

\begin{corollary}[Generalized-projection variational characterization of the selected limit]\label{limitVIchar}
Let the assumptions of Theorem~\ref{mainthm} hold and set
\[
q:=\Pi_Bx_1.
\]
Then $q$ is the unique point in $B$ satisfying
\begin{equation}\label{limitVI}
\langle p-q,Jx_1-Jq\rangle\leq 0,
\qquad \forall p\in B.
\end{equation}
Equivalently, the limit selected by the algorithm is characterized by the generalized projection variational inequality on the full common solution set $B$.
\end{corollary}

\begin{proof}
Since $B$ is nonempty, closed and convex, the generalized projection $\Pi_Bx_1$ is well defined. By Lemma~\ref{projineq}, $q=\Pi_Bx_1$ if and only if \eqref{limitVI} holds. The uniqueness follows from the uniqueness of the generalized projection in uniformly smooth and uniformly convex Banach spaces.
\end{proof}

\begin{corollary}\label{singlecorollary}
Let the assumptions of Theorem~\ref{mainthm} hold. Suppose that $T_n=T$ for all $n\geq 1$, $A_j=A$ for all $j\geq 1$, and $f_i=f$ for all $i\geq 1$. Assume that $T$ is generalized $J_{*}$-nonexpansive and $J_{*}$-closed, and that
\[
B=F_J(T)\cap EP(f)\cap VI(C,A)
\]
is nonempty, closed and convex. Then the sequence $\{x_n\}$ generated by \eqref{algorithm} converges strongly to $\Pi_Bx_1$.
\end{corollary}

\begin{proof}
In this case all selected equilibrium and variational inequality problems are identical, and $T_n=T$ for all $n\geq 1$. Hence the NST-condition is automatic and the conclusion follows from Theorem~\ref{mainthm}.
\end{proof}

\begin{corollary}\label{noEPVIcorollary}
Let $E$, $C$, $T_n$, $\Gamma$, $h$ and $\bar u$ satisfy the fixed point assumptions of Theorem~\ref{mainthm}. Suppose that
\[
B=F_J(\Gamma)
\]
is nonempty, closed and convex. Then the viscosity-Halpern hybrid sequence obtained from \eqref{algorithm} by suppressing the equilibrium and variational inequality steps converges strongly to $\Pi_Bx_1$.
\end{corollary}

\begin{proof}
Take the bifunctions $f_i$ to be identically zero and take $A_j=0$ for all $i,j\geq 1$. Then $EP(f_i)=C$ and $VI(C,A_j)=C$ for all $i,j\geq 1$, and the conclusion follows from Theorem~\ref{mainthm}.
\end{proof}

\begin{corollary}\label{hilbertcorollary}
Let $E=H$ be a real Hilbert space and let $C$ be a nonempty closed and convex subset of $H$. Let $\{f_i\}_{i=1}^{\infty}$ be a countable family of bifunctions from $C\times C$ into $\mathbb{R}$ satisfying $(A1)-(A4)$, let $\{A_j\}_{j=1}^{\infty}$ be a countable family of continuous monotone mappings from $C$ into $H$, and let $\{T_n\}_{n=1}^{\infty}$ be a countable family of nonexpansive-type mappings satisfying the corresponding NST-condition with $\Gamma$. Suppose that
\[
B=F(\Gamma)\cap\left[\bigcap_{i=1}^{\infty}EP(f_i)\right]\cap\left[\bigcap_{j=1}^{\infty}VI(C,A_j)\right],
\]
where \(F(\Gamma):=\bigcap_{T\in\Gamma}F(T)\). Assume that \(B\) is nonempty, closed and convex. Let \(\mu,\nu:\mathbb{N}\to\mathbb{N}\) be index maps whose fibers are infinite, and set \(f_n=f_{\mu(n)}\) and \(A_n=A_{\nu(n)}\). Then the sequence generated by
\[
\left\{
\begin{array}{ll}
s_n=\lambda_nh(x_n)+\sigma_n\bar u+(1-\lambda_n-\sigma_n)x_n,\\[1mm]
z_n=\left\{z\in C:f_n(z,y)+\frac{1}{r_n}\langle y-z,z-s_n\rangle\geq 0,\ \forall y\in C\right\},\\[1mm]
u_n=\left\{z\in C:\langle y-z,A_nz\rangle+\frac{1}{r_n}\langle y-z,z-s_n\rangle\geq 0,\ \forall y\in C\right\},\\[1mm]
y_n=\alpha_1s_n+\alpha_2z_n+\alpha_3T_nu_n,\\[1mm]
C_{n+1}=\{v\in C_n:\|v-y_n\|\leq \|v-s_n\|\},\\[1mm]
x_{n+1}=P_{C_{n+1}}x_1,
\end{array}
\right.
\]
converges strongly to $P_Bx_1$, where $P_B$ is the metric projection of $H$ onto $B$.
\end{corollary}

\begin{proof}
In a Hilbert space, $J=I$ and $\phi(x,y)=\|x-y\|^2$. The generalized projection coincides with the metric projection. Hence the result follows directly from Theorem~\ref{mainthm}.
\end{proof}

\begin{remark}
The theorem and corollaries above are applicable in classical uniformly smooth and uniformly convex Banach spaces, such as $L_p$, $\ell_p$ and $W_p^m(\Omega)$, where $1<p<\infty$, whenever the common solution set is nonempty, closed and convex.
\end{remark}

\begin{proposition}\label{finite-truncation-prop}
The countable-family setting cannot, in general, be replaced by any finite truncation. In particular, there exist countably many variational inequality problems whose full common solution set is a singleton, while every finite truncation has an infinite-dimensional solution set.
\end{proposition}

\begin{proof}
Let $H=\ell_2$ and $C=H$. For each $j\geq 1$, define $A_j:H\to H$ by
\[
A_jx=x_j e_j,
\qquad x=(x_1,x_2,\ldots)\in \ell_2,
\]
where $\{e_j\}$ is the canonical orthonormal basis. Each $A_j$ is continuous and monotone. Moreover,
\[
VI(H,A_j)=\{x\in \ell_2:x_j=0\}.
\]
Therefore
\[
\bigcap_{j=1}^{\infty}VI(H,A_j)=\{0\},
\]
whereas for every finite $N$,
\[
\bigcap_{j=1}^{N}VI(H,A_j)=\{x\in \ell_2:x_1=\cdots=x_N=0\},
\]
which is infinite-dimensional. Hence no finite truncation recovers the full countable intersection.
\end{proof}

\begin{example}\label{l2example}
Let $H=\ell_2$, $C$ be the closed unit ball of $\ell_2$, and let $h(x)=\tau x$ for some $\tau\in(0,1)$. Set $\bar u=0$. Let $T_n=I$ for all $n\geq 1$. For each $j\geq 1$, let $A_jx=x_j e_j$, and for each $i\geq 1$ define
\[
f_i(x,y)=\langle x_i e_i,y-x\rangle,
\qquad x,y\in C.
\]
Then each $A_j$ is continuous and monotone, and each $f_i$ satisfies $(A1)-(A4)$. Moreover,
\[
EP(f_i)=\{x\in C:x_i=0\},
\qquad
VI(C,A_j)=\{x\in C:x_j=0\}.
\]
Thus the common solution set is $B=\{0\}$. If the index maps $\mu$ and $\nu$ visit every index infinitely often, Theorem~\ref{mainthm} implies that the viscosity-Halpern hybrid sequence converges strongly to $0$.
\end{example}

\section{Conclusion}

We introduced a viscosity-Halpern hybrid projection scheme for approximating a common element of the $J$-fixed point set of a countable family of generalized $J_{*}$-nonexpansive mappings, the solution sets of countably many variational inequality problems and the solution sets of countably many equilibrium problems in a uniformly smooth and uniformly convex real Banach space. The proposed algorithm first forms a viscosity-Halpern regularized point and then applies the selected equilibrium and variational inequality resolvent steps before the shrinking generalized projection step is performed.

The main theorem extends finite-family hybrid projection frameworks to a countable-family setting by using index maps which ensure that every component problem is visited infinitely often. We also give a generalized-projection variational characterization of the selected limit, residual convergence, single-problem and fixed point-only consequences, a Hilbert-space version, a comparison with related frameworks and examples showing that finite truncations do not generally recover the full countable problem. These results show that viscosity-Halpern regularization can be incorporated into hybrid projection methods for common solution problems involving countable equilibrium problems, countable variational inequalities and generalized nonexpansive-type mappings.

\end{document}